\documentclass[11pt,reqno]{amsart}
\usepackage{indentfirst, amssymb,amsmath,amsthm}
\usepackage{enumitem}
\usepackage{xcolor}
\usepackage[colorlinks]{hyperref}
\newtheorem{theo}{Theorem}[section]
\newtheorem{exm}{Example}[section]
\newtheorem{lem}{Lemma}

\newtheorem{cor}{Corollary}[section]
\newtheorem{rem}{Remark}[section]

\newtheorem{defi}{Definition}
\newtheorem{ques}{Question}

\newtheorem{note}{Note}

\newcommand{\be}{\begin{equation}}
\newcommand{\ee}{\end{equation}}
\newcommand{\beas}{\begin{eqnarray*}}
\newcommand{\eeas}{\end{eqnarray*}}
\newcommand{\bea}{\begin{eqnarray}}
\newcommand{\eea}{\end{eqnarray}}

\numberwithin{equation}{section}

\begin{document}

\setlength{\unitlength}{1mm} \baselineskip .52cm
%\large
\setcounter{page}{1}
\pagenumbering{arabic}
\title[Common Fixed Point Theorems on Complete and Weak $G$-Complete Fuzzy ...]{Common Fixed Point Theorems on Complete and Weak $G$-Complete Fuzzy Metric Spaces}

\author[Sugata Adhya and A. Deb Ray]{Sugata Adhya and A. Deb Ray}

\address{Department of Mathematics, The Bhawanipur Education Society College. 5, Lala Lajpat Rai Sarani, Kolkata 700020, West Bengal, India.}
\email {sugataadhya@yahoo.com}

\address{Department of Pure Mathematics, University of Calcutta. 35, Ballygunge Circular Road, Kolkata 700019, West Bengal, India.}
\email {debrayatasi@gmail.com}

\maketitle

\begin{abstract}

Motivated by Gopal and Vetro [Iranian Journal of Fuzzy Systems, 11(3), 95-107], we introduce a symmetric pair of $\beta$-admissible mappings and obtain common fixed point theorems for such a pair in complete and weak $G$-complete fuzzy metric spaces. In particular, we rectified, generalize and improve the common fixed point theorem obtained by Turkoglu and Sangurlu [Journal of Intelligent \& Fuzzy Systems, 26(1), 137-142] for two fuzzy $\psi$-contractive mappings. We include non-trivial examples to exhibit the generality and demonstrate our results.

\end{abstract}


\section{\textbf{Introduction}}

Finding an appropriate analogue for metric spaces in fuzzy setting was a long standing problem. In 1975, motivated by the idea of Menger spaces \cite{men}, Kramosil and Michalek \cite{km} gave a solution to this problem and introduced fuzzy metric spaces. Grabiec \cite{g88}, in 1988, defined Cauchy sequences in such spaces. However, while modifying the definition of fuzzy metric, George and Veeramani \cite{ver94} strengthened Grabiec’s definition of Cauchy sequences which is now widely accepted as standard for fuzzy metric spaces.

In fact, Grabiec's original definition of Cauchy sequence (now known as $G$-Cauchy sequence \cite{wg}) was so weak that even a compact fuzzy metric space fails to be complete (now known as $G$-complete \cite{wg}) in the Grabiec's sense. Due to this drawback, in \cite{wg}, Gregori et. al. introduced a new form of completeness. It is called weak $G$-completeness for fuzzy metric spaces. Weak $G$-completeness has been further studied in \cite{saadr1} and \cite{saadr2}. In this paper we establish certain fixed point theorems in weak $G$-complete fuzzy metric spaces.

It is known that, alike metric spaces, fixed point theory is a rich subfield of fuzzy metric spaces where contractive and contractive-type mappings play important roles for obtaining fixed points theorems. The first fuzzy version of Banach Contraction Principle was established in 1988 by Grabiec for $G$-complete fuzzy metric spaces \cite{g88}. In 2002, Gregori and Sapena introduced fuzzy contractive mappings and obtained several fixed point theorems for complete fuzzy metric spacesi \cite{gs2}. Mihet enlarged this class of contractive mappings and introduced the notion of fuzzy $\psi$-contractive mappings. This new class of mappings was utilized to establish a new version of fuzzy Banach contraction theorem for complete non-Archimedean fuzzy metric spaces \cite{mi} which was further generalized for weak $G$-complete fuzzy metric spaces \cite{wg}.

The above class of contractive mappings, introduced by Mihet, has been extensively used to obtain fixed point theorems in fuzzy metric spaces. In 2014, Turkoglu and Sangurlu obtained a common fixed point theorem for a pair of fuzzy $\psi$-contractive mappings in $G$-complete fuzzy metric spaces \cite{ts}.

In this paper, motivated by the work of Gopal and Vetro \cite{gv}, we introduce a symmetric pair of $\beta$-admissible mappings and a pair of $\beta$-$\psi$-fuzzy contractive mappings. These new families are utilized here to establish common fixed point theorems in complete and weak $G$-complete fuzzy metric spaces, both in the sense of \cite{ver94} and \cite{km}. In particular, we rectified the fixed point theorem obtained by Turkoglu et. al. \cite{ts} and substantially generalize and improve it. Our theory is supported and illustrated by appropriate examples.

\section{\textbf{Preliminaries}}

In this section, we recall some basic definitions and facts which are referred subsequently. 

\begin{defi}
\normalfont\cite{ss} A mapping $*:[0,1]\times[0,1]\to[0,1]$ is called a continuous $t$-norm if (i) $*$ is associative and commutative, (ii) $*$ is continuous, (iii) $a*1=a,~\forall~a\in[0,1],$ and (iv) for $a,b,c,d\in[0,1],$ $a\le c,b\le d\implies a*b\le c*d.$\\

It is easy to note that, the followings are examples of continuous $t$-norms:

(i)  $a*b=\min(a,b),$ and 

(ii) $a*b=ab$ 

for any $a,b\in[0,1],$ 
\end{defi}

\begin{defi}
\normalfont(Kramosil and Michalek \cite{km}) Given a nonempty set $X,$ a continuous $t$-norm $*$ and a mapping $M:X\times X\times[0,\infty)\to[0,1],$ the ordered triple $(X,M,*)$ is called a KM fuzzy metric space if, for all $x,y,z\in X$ and $s,t>0,$ the following conditions hold:

a) $M(x,y,0)=0,$

b) $M(x,y,t)=1,~\forall~t>0\iff x=y,$

c) $M(x,y,t)=M(y,x,t),$

d) $M(x,y,t)*M(y,z,s)\le M(x,z,t+s),$

e) $M(x,y,.):[0,\infty)\to[0,1]$ is left continuous.
\end{defi}

\begin{defi}
\normalfont(George and Veeramani \cite{ver94}) Given a nonempty set $X,$ a continuous $t$-norm $*$ and a mapping $M:X\times X\times(0,\infty)\to[0,1],$ the ordered triple $(X,M,*)$ is called a GV fuzzy metric space if, for all $x,y,z\in X$ and $s,t>0,$ the following conditions hold:

a) $M(x,y,t)>0,$

b) $M(x,y,t)=1\iff x=y,$

c) $M(x,y,t)=M(y,x,t),$

d) $M(x,y,t)*M(y,z,s)\le M(x,z,t+s),$

e) $M(x,y,.):(0,\infty)\to[0,1]$ is continuous.

Unless otherwise specified, by a fuzzy metric space we refer to the GV fuzzy metric space.
\end{defi}

\begin{defi}\cite{is}
\normalfont A (KM) fuzzy metric space $(X,M,*)$ is said to be non-Archimedean if $M(x,y,t)*M(y,z,s)\le M(x,z,\max\{t,s\}),$ for all $x,y,z\in X$ and $s,t>0.$
\end{defi}

\begin{lem}
\normalfont(\cite{ver94}, \cite{g88}) Given a (KM) fuzzy metric space $(X,M,*),$ $M(x,y,\cdot)$ defines a nondecreasing map, $\forall~x,y\in X.$
\end{lem}

Let $(X,M,*)$ be a (KM) fuzzy metric space. It is well-known (\cite{ver94}, \cite{wg}) that, $\{B(x,r,t):x\in X,r\in(0,1),t>0\}$ forms a base for some topology $\tau_M$ on $X,$ where $B(x,r,t)=\{y\in X:M(x,y,t)>1-r\},~\forall~x\in X,r\in(0,1),t>0$. The topological behaviour of $(X,M,*)$ are defined with respect to the topology $\tau_M.$ In particular, given two (KM) fuzzy metric spaces $(X,M,*)$ and $(Y,N,\star),$ a mapping $f:X\to Y$ is called continuous if $f$ is continuous as a mapping from $(X,\tau_M)$ to $(Y,\tau_N).$

Similarly, sequential convergence are defined as follows: A sequence $(x_n)$ in a (KM) fuzzy metric space $(X,M,*)$ is said to be convergent to some $x\in X$ (\textit{resp.} clusters), if it does so in $(X,\tau_M).$

It is easy to note that, if $f$ is a continuous mapping from a (KM) fuzzy metric space $(X,M,*)$ to a (KM) fuzzy metric space $(Y,N,\star),$ and $(x_n)$ is a sequence in $X$ converging to $x\in X,$ then $f(x_n)\to f(x)$ as $n\to\infty.$

The following is an easy consequence that has been shown in \cite{ver94} for GV fuzzy metric spaces. The case for KM fuzzy metric spaces is similar as stated next.

\begin{theo}
\normalfont A sequence $(x_n)$ in a (KM) fuzzy metric space $(X,M,*)$ converges to $x\in X$ if and only if $\lim\limits_{n\to\infty}M(x_n,x,t)= 1,~\forall~t>0.$
\end{theo}

\begin{defi}
\normalfont (\cite{ver94}, \cite{s66}) Let $(X,M,*)$ be a (KM) fuzzy metric space. A sequence $(x_n)$ in $X$ is called Cauchy if for $\epsilon\in(0,1),t>0,$ there exists $k\in\mathbb N$ such that  $M(x_m,x_n,t)>1-\epsilon,~\forall~m,n\ge k.$ Clearly every convergent sequence in $(X,M,*)$ is Cauchy. $(X,M,*)$ is complete if every Cauchy sequence in it converges.
\end{defi}

\begin{defi}
\normalfont (\cite{g88}, \cite{wg}) Let $(X,M,*)$ be a (KM) fuzzy metric space. A sequence $(x_n)$ in $X$ is called $G$-Cauchy if $\lim\limits_{n\to\infty}M(x_n,$ $x_{n+1},t)=1,~\forall~t>0.$ If, in $(X,M,*),$ every $G$-Cauchy sequence converges, then $(X,M,*)$ is said to be $G$-complete.

Though $G$-completeness necessarily imply completeness in (KM) fuzzy metric spaces, a compact (KM) fuzzy metric space may not be $G$-complete.
To overcome this drawback, the following weaker version of completeness has been introduced in \cite{wg} .

A (KM) fuzzy metric space in which every $G$-Cauchy sequence clusters is called a weak $G$-complete (KM) fuzzy metric space.
\end{defi}

We finish this section by recalling the definitions of $\beta$-$\psi$-fuzzy contractive mapping and $\beta$-admissible mapping introduced by Gopal and Vetro in \cite{gv}.

Following \cite{mi}, we denote by $\Psi$ the family of all mappings $\psi:[0,1]\to[0,1]$ such that,

(i) $\psi$ is non-decreasing and continuous, 

(ii) $\psi(t)>t,~\forall~t\in(0,1).$ 

It is easy to check that $\psi(1)=1$ and $\lim\limits_{n\to\infty}\psi^n(r)=1,~\forall~\psi\in\Psi,r\in(0,1)$ (e.g. consult \cite{cv11}). 

\begin{defi}
\normalfont Let $(X,M,*)$ be a (KM) fuzzy metric space and $f$ a self-mapping on $X$. For some $\psi\in\Psi,$ and a mapping $\beta:X^2\times(0,\infty)\to(0,\infty),$  $f$ is called a $\beta$-$\psi$-fuzzy contractive mappings if $\forall~x,y\in X$ with $x\ne y$ and $t>0,$ $$M(x,y,t)>0\implies\beta(x,y,t)M(fx,fy,t)\ge\psi(M(x,y,t)).$$\end{defi}

\begin{defi}
\normalfont Let $(X,M,*)$ be a (KM) fuzzy metric space and $f$ a self-mapping on $X$. For a mapping $\beta:X^2\times(0,\infty)\to(0,\infty),$  $f$ is called $\beta$-admissible if $\forall~x,y\in X,t>0,$ $$\beta(x,y,t)\le1\implies\beta(fx,fy,t)\le1.$$
\end{defi}

\section{\textbf{Main Results}}

We begin this section by introducing a pair of $\beta$-$\psi$-fuzzy contractive mappings and a symmetric pair of $\beta$-admissible mappings that extend respectively the class of $\beta$-$\psi$-fuzzy contractive mappings and $\beta$-admissible mappings in a (KM) fuzzy metric space. 

\begin{defi}
\normalfont(\cite{mi}, \cite{cv11}) Let $(X,M,*)$ be a (KM) fuzzy metric space and $\psi\in\Psi.$ 

a) A mapping $f:X\to X$ is called fuzzy $\psi$-contractive if $\forall~x,y,\in X,t>0,$ $$M(x,y,t)>0\implies M(fx,fy,t)\ge\psi(M(x,y,t)).$$

b) Given a pair of mappings $f,g:X\to X,$ $(f,g)$ is called a pair of fuzzy $\psi$-contractive mappings if $\forall~x,y,\in X,t>0,$ $$M(x,y,t)>0\implies M(fx,gy,t)\ge\psi(\min\{M(x,y,t),M(x,fx,t),M(y,gy,t)\}).$$
\end{defi}

\begin{defi}
\normalfont Let $(X,M,*)$ be a (KM) fuzzy metric space and $f,g$ be self-mappings on $X$. For a mapping $\beta:X^2\times(0,\infty)\to(0,\infty),$  the pair $(f,g)$ is called a symmetric pair of $\beta$-admissible mappings if $\forall~x,y\in X,t>0,$ $$\beta(x,y,t)\le1\implies\max\{\beta(fx,gy,t),\beta(gy,fx,t),\beta(gx,fy,t),\beta(fy,gx,t)\}\le1.$$ 
\end{defi}

We note that if  $\beta(x,y,t)\equiv1,$ then a pair of $\beta$-$\psi$-fuzzy contractive mappings is a pair of $\psi$-contractive mappings.

\begin{theo}\label{res3}
\normalfont Let $(X,M,*)$ be a weak $G$-complete (KM) fuzzy metric space and $f,g$ be continuous self-mappings on $X$ satisfying $M(x,fx,t),M(x,gx,t)>0,$ $\forall~x\in X,t>0.$ Suppose for some $\psi\in\Psi,$ and a mapping $\beta:X^2\times(0,\infty)\to(0,\infty),$ $(f,g)$ is a pair of $\beta$-$\psi$-fuzzy contractive mappings such that

(i) $(f,g)$ is a symmetric pair of $\beta$-admissible mappings,

(ii) for some $x_0\in X,$ $\beta(x_0,fx_0,t)\le1$ and $M(x_0,fx_0,t)>0,$ $\forall~t>0.$

Then $f$ and $g$ have a common fixed point in $X.$ 

Moreover, if $\beta(x,y,t)\le1,~\forall~x,y\in X,t>0$ and for $x,y~(x\ne y)\in X,$ $M(x,y,t)>0,~\forall~t>0,$ then the fixed point is unique.
\end{theo}

\begin{proof}
\normalfont Define a sequence $(x_n)$ as follows: $x_1=fx_0,x_2=gx_1,\cdots,x_{2n+1}=fx_{2n},x_{2n+2}=gx_{2n+1},\cdots$.

We have $\beta(x_0,x_1,t)=\beta(x_0,fx_0,t)\le1,$ $\forall~t>0.$ Since $\beta(x_{n-1},x_n,t)\le1\implies\beta(x_n,x_{n+1},t)\le\max\{\beta(fx_{n-1},gx_n,t),\beta(gx_{n-1},fx_n,t)\}\le1,$ $\forall~n\in\mathbb N,t>0,$ so by applying induction, we obtain $\beta(x_n,x_{n+1},t)\le1,~\forall~n=0,1,2,\cdots$ and $t>0.$ 

Consequently, $\beta(x_{n+1},x_n,t)\le1,~\forall~n\in\mathbb N,t>0.$

Let $n$ be a positive integer.

If $n$ is $\text{odd}$, then for chosen $t>0,$\\$M(x_{n-1},x_n,t)>0\implies M(x_n,x_{n+1},t)=M(fx_{n-1},gx_n,t)\\\ge\beta(x_{n-1},x_n,t)M(fx_{n-1},gx_n,t)\\\ge\psi(\min\{M(x_{n-1},x_n,t),M(x_{n-1},fx_{n-1},t),M(x_n,gx_n,t)\})\\=\psi(M(x_{n-1},x_n,t))>0.$ 

Again if $n$ is $\text{even},$ then for chosen $t>0,$ \\$M(x_{n-1},x_n,t)>0\implies M(x_n,x_{n+1},t)=M(gx_{n-1},fx_n,t)\\=M(fx_n,gx_{n-1},t)\\\ge\beta(x_{n},x_{n-1},t)M(fx_n,gx_{n-1},t)\\\ge\psi(\min\{M(x_{n},x_{n-1},t),M(x_{n},fx_{n},t),M(x_{n-1},gx_{n-1},t)\})\\=\psi(M(x_{n-1},x_n,t))>0.$

Thus, by using induction, we have $M(x_n,x_{n+1},t)\ge\psi(M(x_{n-1},x_n,t)),$ $\forall~n\in\mathbb N,t>0$ and hence, $M(x_n,x_{n+1},t)\ge\psi^n(M(x_0,x_1,t)),$ $\forall~n\in\mathbb N,t>0.$

Thus $\lim\limits_{n\to\infty}M(x_n,x_{n+1},t)=1,~\forall~t>0.$ Hence $(x_n)$ is a $G$-Cauchy sequence in $X$.

Since $(X,M,*)$ is weak $G$-complete, $(x_n)$ has a cluster point $x$ in $X.$

So there exists a subsequence $(x_{r_n})$ of $(x_n)$ such that $(x_{r_n})$ converges to $x$ where $r_n$'s are either all even or all odd.

Without loss of generality, suppose all the $r_n$'s are even. Then $\forall~n\in\mathbb N,t>0,$ 

\noindent$M(x,fx,t)\ge M\left(x,x_{r_n},\frac{t}{3}\right)*M\left(x_{r_n},x_{r_n+1},\frac{t}{3}\right)*M\left(x_{r_n+1},fx,\frac{t}{3}\right)$

\noindent$=M\left(x,x_{r_n},\frac{t}{3}\right)*M\left(x_{r_n},x_{r_n+1},\frac{t}{3}\right)*M\left(fx_{r_n},fx,\frac{t}{3}\right)$ 

Taking limit as $n\to\infty,$ we have $M(x,fx,t)=1,~\forall~t>0,$ since $f$ is continuous and $(x_n)$ is $G$-Cauchy. Hence $fx=x.$

Again, since $\forall~t>0,$ $M(x_{r_n-1},x,t)\ge M\left(x_{r_n-1},x_{r_n},\frac{t}{2}\right)*M\left(x_{r_n},x,\frac{t}{2}\right)\to1*1=1$ as $n\to\infty,$ so $\forall~t>0,$ $M(x_{r_n-1},x,t)\to1$ as $n\to\infty,$ and consequently, $\forall~t>0,$ $M(x,gx,t)\ge M\left(x,x_{r_n},\frac{t}{2}\right)*M\left(x_{r_n},gx,\frac{t}{2}\right)=M\left (x,x_{r_n},\frac{t}{2}\right)*M\left(gx_{r_n-1},gx,\frac{t}{2}\right)\to1*1=1$ as $n\to\infty,$ since $g$ is continuous.

Thus $M(x,gx,t)=1,~\forall~t>0,$ and hence $gx=x.$

Again, if all the $r_n$'s are odd, it can be similarly shown that $fx=gx=x.$

Thus, in general, $x$ is a common fixed point of both $f$ and $g.$

Let us now assume, in addition to condition (i) and (ii), that $\beta(x,y,t)\le1,~\forall~x,y\in X,t>0$ and for $x,y~(x\ne y)\in X,$ $M(x,y,t)>0,$ $\forall~t>0.$

If possible, let $x,y$ be two common fixed points of $f$ and $g$ such that $x\ne y.$ Then there exists $t_0>0$ such that $0<M(x,y,t_0)<1.$ 

Now $M(x,y,t_0)\\\ge\beta(x,y,t_0)M(fx,gy,t_0)\\\ge\psi(\min\{M(x,y,t_0),M(x,fx,t_0),M(y,gy,t_0)\})\\=\psi(M(x,y,t_0))\\>M(x,y,t_0),$ a contradiction. 

Hence the common fixed point is unique.
\end{proof}

By putting $\beta\equiv1$ in Theorem \ref{res3}, we have the following:

\begin{cor}\label{corts}
\normalfont Let $(X,M,*)$ be a weak $G$-complete (KM) fuzzy metric space and $f,g$ be continuous self-mappings on $X$ satisfying $M(x,fx,t),M(x,gx,t)>0,$ $\forall~x\in X,t>0.$ Suppose for some $\psi\in\Psi$ and $x_0\in X,$ $M(x_0,fx_0,t)>0,$ $\forall~t>0$ and $M(fx,gy,t)\ge\psi(\min\{M(x,y,t),M(x,fx,t),M(y,gy,t)\}),$ $\forall~x,y\in X,t>0.$ Then $f$ and $g$ have a common fixed point in $X.$ 

Moreover, if for $x,y~(x\ne y)\in X,$ $M(x,y,t)>0,~\forall~t>0,$ then the fixed point is unique.
\end{cor}

The last corollary shows that in the following common fixed point result obtained by Turkoglu and Sangurlu, we do not require to specify the continuous $t$-norm $*$ or assume $f,g$ to be fuzzy $\psi$-contractive mappings.

\begin{theo}\label{cort}
\normalfont (\cite{ts}, \textit{Common fixed point theorem by Turkoglu and Sangurlu}) Let $(X,M,*)$ be a $G$-complete KM fuzzy metric space such that $a*b=\min\{a,b\}.$ Let there is a $x_0\in X$ such that $M(x_0,fx_0,t)>0,~\forall~t>0.$ If $f$ and $g$ are continuous self-mappings on $X$ satisfying $M(x,fx,t),M(x,gx,t)>0,$ $\forall~x\in X,t>0$ such that for some $\psi\in\Psi,$

(i) $f,g$ are fuzzy $\psi$-contractive mappings,

(ii) $M(fx,gy,t)\ge\psi(\min\{M(x,y,t),M(x,fx,t),M(y,gy,t)\}),~\forall~x,y\in X,t>0,$

Then $f$ and $g$ have a common fixed point in $X.$
\end{theo}

\begin{note}
\normalfont It should be noted, in \cite{ts}, Turkoglu et. al. included the following additional condition in the hypothesis: 

\textit{$(x_n)$ is a sequence in $X$ such that $x_1=fx_0,x_2=gx_1,\cdots,x_{2n+1}=f(x_{2n}),x_{2n+2}=g(x_{2n+1}),\cdots.$}  

However, noting that such sequence always exists, we remove it from the statement of Theorem \ref{cort}. 

Moreover, we have included the additional condition: 

\textit{$M(x,fx,t),M(x,gx,t)>0,$ $\forall~x\in X,t>0$}

in the hypothesis of Theorem \ref{cort} noting that it is essential even for the proof provided by Turkoglu et. al. \cite{ts}.

\end{note}

\begin{note}
\normalfont In \cite{ts}, Turkoglu et. al. concluded that under the hypothesis of Theorem \ref{cort}, $f$ and $g$ have a \textit{unique} common fixed point. However this is not true as is exhibited next. 

Consider the KM fuzzy metric space $(X,M,*),$ where $X=[0,1],$ $a*b=\min\{a,b\},~\forall~a,b\in[0,1],$ and for $x,y\in X,$ $$M(x,y,t)=\begin{cases}0& \text{if }x\ne y,t\ge0\\\noalign{\vskip9pt}0& \text{if }x=y,t=0\\\noalign{\vskip9pt}1& \text{if }x=y,t>0\end{cases}$$

Then $X$ is a $G$-complete KM fuzzy metric space. 

Then by setting $\psi(t)=\sqrt{t},~\forall~t\in[0,1],$ for $f,g:X\to X$ given by $f(x)=g(x)=x,~\forall~x\in X$ the hypothesis of Theorem \ref{cort} get satisfied. However every point of $X$ is a common fixed point of $f$ and $g$ and hence the common fixed point is not unique.
\end{note}

In what follows, we exhibit the applicability of our theorem, namely, Theorem \ref{res3} over Theorem \ref{cort} (due to \cite{ts}) in the following aspects:

(a) Theorem \ref{res3} is applicable even for a weak $G$-complete (KM) fuzzy metric space $(X,M,*)$ which is not $G$-complete;

(b) Theorem \ref{res3} is applicable also for $\beta\not\equiv1;$

(c) Theorem \ref{res3} is applicable even when $a*b$ is not defined as $\min\{a,b\}.$

\begin{exm}\label{ex33}
\normalfont Consider the fuzzy metric space $(X,M,\cdot),$ where $X=\left\{\frac{1}{2^n}:n\ge2\right\}$ $\bigcup\left[\frac{1}{2},1\right],$ $M(x,y,t)=\frac{\min\{x,y\}}{\max\{x,y\}},~\forall~x,y\in X,t>0$ and $\cdot$ is the usual product of reals. It is known that $(X,M,\cdot)$ is a weak $G$-complete fuzzy metric space which is not $G$-complete \cite{wg}. Further, $\tau_M$ defines the usual topology of $\mathbb R$ restricted to the set $X$ \cite{st}.

Let $f,g$ be two self-mappings on $X$ such that $fx=x,$ $gx=1,~\forall~x\in X.$ Then $f,g$ are continuous on $X.$

Now by setting $\beta:X^2\times(0,\infty)\to(0,\infty)$ as $$\beta(x,y,t)=\begin{cases}\frac{1}{x}& \text{if }x\le y\\\noalign{\vskip9pt}\frac{1+x}{2x}& \text{if }y<x<1\\\noalign{\vskip9pt}2& \text{if }y<x=1\end{cases}$$ for all $x,y\in X,t>0,$ we see that $(f,g)$ defines a symmetric pair of $\beta$-admissible mappings and $\forall~t>0,$ $\beta(x_0,fx_0,t)\le1,M(x_0,fx_0,t)>0$ hold for $x_0=1.$

Let $\psi:[0,1]\to[0,1]$ be defined by $\psi(x)=\frac{1+x}{2},~\forall~x\in[0,1].$ Clearly, $\psi\in\Psi.$ 

We now show that, $(f,g)$ is a pair of $\beta$-$\psi$-fuzzy contractive mappings.

Choose $x,y\in X$ and $t>0.$

\textit{Case I:} $x\le y.$ Then $\beta(x,y,t)M(fx,gy,t)=\frac{1}{x}\times M(x,1,t)=1\\\ge\psi(\min\{M(x,y,t),M(x,fx,t),M(y,gy,t)\}).$

\textit{Case II:} $y<x<1.$ Then $\beta(x,y,t)M(fx,gy,t)=\frac{1+x}{2}$ and

$\psi(\min\{M(x,y,t),M(x,fx,t),M(y,gy,t)\})=\psi(\min\{\frac{y}{x},1,y\})=\psi(y)=\frac{1+y}{2}.$

Since $x>y,$ we obtain

$\beta(x,y,t)M(fx,gy,t)\ge\psi(\min\{M(x,y,t),M(x,fx,t),M(y,gy,t)\}).$

\textit{Case III:} $y<x=1.$ Then $\beta(x,y,t)M(fx,gy,t)=2\times M(x,1,t)\ge\psi(\min\{M(x,y,t),\\M(x,fx,t),M(y,gy,t)\}).$

Thus, in view of Theorem \ref{res3}, $f,g$ have a common fixed point.
\end{exm}

In what follows we show that, in Theorem \ref{res3}, the last condition is essential to ensure the uniqueness of the common fixed point.

\begin{exm}\label{ex1}
\normalfont Consider the fuzzy metric space $(X,M,\cdot)$ of Example \ref{ex33}. Let $f,g$ be self-mappings on $X$ such that $$fx=gx=\begin{cases}\frac{1}{4}&\text{if }x=\frac{1}{4}\\\noalign{\vskip9pt}1&\text{otherwise}\end{cases}.$$ 

Clearly, $f,g$ are continuous on $X.$ 

If $\beta:X^2\times(0,\infty)\to(0,\infty)$ is defined by $$\beta(x,y,t)=\begin{cases}1&\text{if }x=y=1\\\noalign{\vskip9pt}4&\text{otherwise}\end{cases}$$ then $(f,g)$ is a pair of $\beta$-$\psi$-fuzzy contractive mappings, $\forall~\psi\in\Psi.$

It is easy to see that conditions (i) and (ii) of Theorem \ref{res3} hold immediately. 

So, in view of Theorem \ref{res3}, $f,g$ have a common fixed point. 

However the fixed point is not unique. Indeed $x=1$ and $x=\frac{1}{4}$ are two such points.
\end{exm}

\begin{cor}
\normalfont Let $(X,M,*)$ be a weak $G$-complete (KM) fuzzy metric space and $f$ be a continuous self-mapping on $X$ satisfying $M(x,fx,t)>0,$ $\forall~x\in X,t>0.$ Suppose for some $\psi\in\Psi,$ and a mapping $\beta:X^2\times(0,\infty)\to(0,\infty),$ $(f,f)$ is a pair of $\beta$-$\psi$-fuzzy contractive mappings such that

(i) $(f,f)$ is a symmetric pair of $\beta$-admissible mappings,

(ii) for some $x_0\in X,$ $\beta(x_0,fx_0,t)\le1$ and $M(x_0,fx_0,t)>0,$ $\forall~t>0.$

Then $f$ has a fixed point in $X.$

Moreover, if $\beta(x,y,t)\le1,~\forall~x,y\in X,t>0$ and for $x,y~(x\ne y)\in X,$ $M(x,y,t)>0,~\forall~t>0,$ then the fixed point is unique.
\end{cor}

\begin{theo}\label{thc32}
\normalfont Let $(X,M,*)$ be a weak $G$-complete (KM) fuzzy metric space and $f,g$ be continuous self-mappings on $X$ satisfying $M(x,fx,t),M(x,gx,t)>0,$ $\forall~x\in X,t>0.$ Suppose for some $\psi_1,\psi_2,\psi_3\in\Psi,$ and a mapping $\beta:X^2\times(0,\infty)\to(0,\infty),$ $M(x,y,t)>0\implies$ $$\beta(x,y,t)M(fx,gy,t)\ge\psi_1(M(x,y,t))+\psi_2(M(x,fx,t))+\psi_3(M(y,gy,t)),$$ $\forall~x,y\in X,t>0$ such that

(i) $(f,g)$ is a symmetric pair of $\beta$-admissible mappings,

(ii) for some $x_0\in X,$ $\beta(x_0,fx_0,t)\le1$ and $M(x_0,fx_0,t)>0,$ $\forall~t>0.$

Then $f$ and $g$ have a common fixed point in $X.$ 

Moreover, if $\beta(x,y,t)\le1,~\forall~x,y\in X,t>0$ and for $x,y~(x\ne y)\in X,$ $M(x,y,t)>0,~\forall~t>0,$ then the fixed point is unique.
\end{theo}

\begin{proof}
Set $\psi=\min\{\psi_1,\psi_2,\psi_3\}.$ Then it is easy to see that $\psi\in\Psi.$

Further $\forall~x,y\in X,t>0,$

$M(x,y,t)>0\implies\beta(x,y,t)M(fx,gy,t)\\\ge\psi_1(M(x,y,t))+\psi_2(M(x,fx,t))+\psi_3(M(y,gy,t))\\\ge\psi(\min\{M(x,y,t),M(x,fx,t),M(y,gy,t)\}).$

Thus the conclusion follows from Theorem \ref{res3}.
\end{proof}

\begin{cor}
\normalfont Let $(X,M,*)$ be a weak $G$-complete (KM) fuzzy metric space and $f,g$ be continuous self-mappings on $X$ satisfying $M(x,fx,t),M(x,gx,t)>0,$ $\forall~x\in X,t>0.$ Suppose for some $a,b,c>1,$ and a mapping $\beta:X^2\times(0,\infty)\to(0,\infty),$ $M(x,y,t)>0\implies$ 
\begin{center}
$\beta(x,y,t)M(fx,gy,t)\ge 
aM(x,y,t)+bM(x,fx,t)+cM(y,gy,t),$    
\end{center}
$\forall~x,y\in X,t>0$ such that

(i) $(f,g)$ is a symmetric pair of $\beta$-admissible mappings,

(ii) for some $x_0\in X,$ $\beta(x_0,fx_0,t)\le1$ and $M(x_0,fx_0,t),$ $\forall~t>0.$

Then $f$ and $g$ have a common fixed point in $X.$ 

Moreover, if $\beta(x,y,t)\le1,~\forall~x,y\in X,t>0$ and for $x,y~(x\ne y)\in X,$ $M(x,y,t)>0,~\forall~t>0,$ then the fixed point is unique.
\end{cor}

\begin{proof}
Immediate from Theorem \ref{thc32} by setting 
\begin{center}
 $\psi_1(x)=\min\{ax,\sqrt{x}\},$\\$\psi_2(x)=\min\{bx,\sqrt{x}\},$\\$\psi_3(x)=\min\{cx,\sqrt{x}\},$   
\end{center}
$\forall~x\in[0,1].$ 
\end{proof}

In the next theorem, we omit the continuity hypothesis of $f,g$ in Theorem \ref{res3} for fuzzy metric spaces defined by George and Veeramani \cite{ver94}.

\begin{theo}\label{th32}
\normalfont Let $(X,M,*)$ be a non-Archimedean weak $G$-complete fuzzy metric space and $f,g$ be self-mappings on $X.$ Suppose for some $\psi\in\Psi,$ and a mapping $\beta:X^2\times(0,\infty)\to(0,\infty),$ $(f,g)$ is a pair of $\beta$-$\psi$-fuzzy contractive mappings such that

(i) $(f,g)$ is a symmetric pair of $\beta$-admissible mappings,

(ii) for some $x_0\in X,$ $\beta(x_0,fx_0,t)\le1$ $\forall~t>0,$

(iii) for each sequence $(x_n)$ in $X$ with $\beta(x_n,x_{n+1},t)\le1,~\forall~n\in\mathbb N,t>0$ and a subsequence $(x_{r_n})$ of $(x_n)$ with $\lim\limits_{n\to\infty}x_{r_n}=x,$ we have $\beta(x_{r_n},x,t)\le1,~\forall~n\in\mathbb N,t>0.$

Then $f$ and $g$ have a common fixed point in $X.$ 

Moreover, if $\beta(x,y,t)\le1,~\forall~x,y\in X,t>0,$ then the fixed point is unique.
\end{theo}

\begin{proof}
Proceeding as in Theorem \ref{res3}, we obtain the $G$-Cauchy sequence $(x_n)$ in $X$ given by $x_1=fx_0,x_2=gx_1,\cdots,x_{2n+1}=fx_{2n},x_{2n+2}=gx_{2n+1},\cdots$ such that $\beta(x_n,x_{n+1},t)\le1,~\forall~n=0,1,2,\cdots$ and $t>0.$

Since $(X,M,*)$ is weak $G$-complete, there exists a subsequence $(x_{r_n})$ of $(x_n)$ such that $(x_{r_n})$ converges to some $x\in X$ where $r_n$'s are either all even or all odd.

Without loss of generality, suppose all the $r_n$'s are even.

If possible, let $gx\ne x.$ Then $0<M(x,gx,t_0)<1,$ for some $t_0>0.$

Since $X$ is non-Archimedean, $\forall~n\in\mathbb N$ we have 

$M(x,gx,t_0)\ge M(x,x_{r_n+1},t_0)*M(x_{r_n+1},gx,t_0)\\=M(x,x_{r_n+1},t_0)*M(fx_{r_n},gx,t_0)\\\ge M(x,x_{r_n+1},t_0)*(\beta(x_{r_n},x,t_0)M(fx_{r_n},gx,t_0))\text{(by (iii))}\\\ge M(x,x_{r_n+1},t_0)*\psi(\min\{M(x_{r_n},x,t_0),M(x_{r_n},fx_{r_n},t_0),M(x,gx,t_0)\})\\\ge M(x,x_{r_n},t_0)*M(x_{r_n},x_{r_n+1},t_0)*\psi(\min\{M(x_{r_n},x,t_0),M(x_{r_n},x_{r_n+1},t_0),M(x,gx,t_0)\})\\\to1*1*\psi(M(x,gx,t_0))=\psi(M(x,gx,t_0))$ as $n\to\infty,$ since $*$ is continuous and $(x_n)$ is $G$-Cauchy.

Thus $M(x,gx,t_0)\ge\psi(M(x,gx,t_0))>M(x,gx,t_0),$ a contradiction.

Consequently $gx=x.$

Again if $fx\ne x,$ $0<M(x,fx,t_1)<1,$ for some $t_1>0.$

Since $X$ is non-Archimedean, $\forall~n\in\mathbb N$ we have 

$M(x,fx,t_1)\ge M(x,x_{r_n+2},t_1)*M(x_{r_n+2},fx,t_1)\\= M(x,x_{r_n+2},t_1)*M(fx,gx_{r_n+1},t_1)\\\ge M(x,x_{r_n+1},t_1)*M(x_{r_n+1},x_{r_n+2},t_1)*M(fx,gx_{r_n+1},t_1).$

Since $\beta(x_{r_n},x,t_1)\le1,~\forall~n\in\mathbb N$ we have $\beta(x,x_{r_n+1},t_1)\le1,~\forall~n\in\mathbb N$ by (i).

Thus $\forall~n\in\mathbb N,$ 

$M(x,fx,t_1)\ge M(x,x_{r_n+1},t_1)*M(x_{r_n+1},x_{r_n+2},t_1)*(\beta(x,x_{r_n+1},t_1)M(fx,gx_{r_n+1},t_1))\\\ge M(x,x_{r_n+1},t_1)*M(x_{r_n+1},x_{r_n+2},t_1)*\psi(\min\{M(x,x_{r_n+1},t_1),M(x,fx,t_1),M(x_{r_n+1},\\gx_{r_n+1},t_1)\})\\\ge M(x,x_{r_n+1},t_1)*M(x_{r_n+1},x_{r_n+2},t_1)*\psi(\min\{M(x,x_{r_n},t_1)*M(x_{r_n},x_{r_n+1},t_1),M(x,fx,\\t_1),M(x_{r_n+1},gx_{r_n+1},t_1)\})\\\ge M(x,x_{r_n+1},t_1)*M(x_{r_n+1},x_{r_n+2},t_1)*\psi(\min\{M(x,x_{r_n},t_1)*M(x_{r_n},x_{r_n+1},t_1),M(x,fx,\\t_1),M(x_{r_n+1},x_{r_n+2},t_1)\})\to\psi(M(x,fx,t_1))$ as $n\to\infty,$ since $*$ is continuous and $(x_n)$ is $G$-Cauchy.

Thus $M(x,fx,t_1)\ge\psi(M(x,fx,t_1))>M(x,fx,t_1),$ a contradiction.

Consequently $fx=x.$

Thus $f,g$ have a common fixed point $x.$

Again, if all the $r_n$'s are odd, it can be similarly shown that $fx=gx=x.$

Uniqueness of the fixed point can be proved by proceeding as in Theorem \ref{res3}.
\end{proof}

\begin{exm}
\normalfont Consider the fuzzy metric space $(X,M,\cdot)$ of Example \ref{ex33}. Let $f,g$ be self-mappings on $X$ such that $$fx=\begin{cases}1&\text{if }x\in\mathbb Q\\\noalign{\vskip9pt}\frac{1}{\sqrt{2}}&\text{otherwise}\end{cases}$$ and $$gx=\begin{cases}1&\text{if }x=1\\\noalign{\vskip9pt}\frac{1}{\sqrt{2}}&\text{if }x=\frac{1}{\sqrt{2}}\\\noalign{\vskip9pt}\frac{1}{2}&\text{otherwise}\end{cases}$$

If $\beta:X^2\times(0,\infty)\to(0,\infty)$ is defined by $$\beta(x,y,t)=\begin{cases}1&\text{if }x=y=1\\\noalign{\vskip9pt}4&\text{otherwise}\end{cases}$$ then $(f,g)$ is a pair of $\beta$-$\psi$-fuzzy contractive mappings, $\forall~\psi\in\Psi.$

It is easy to see that conditions (i)$-$(ii) of Theorem \ref{th32} hold immediately. 

So, in view of Theorem \ref{th32}, $f,g$ have a common fixed point. 

However the fixed point is not unique. Indeed $x=1$ and $x=\frac{1}{\sqrt{2}}$ are two such points.
\end{exm}

By putting $\beta\equiv1$ in Theorem \ref{th32} we see that for a GV fuzzy metric space, non-Archimedeaness replaces the continuity of $f,g$ in Corollary \ref{corts}:

\begin{cor}\label{cor35}
\normalfont \normalfont Let $(X,M,*)$ be a non-Archimedean weak $G$-complete fuzzy metric space and $f,g$ be self-mappings on $X.$ Suppose for some $\psi\in\Psi,$ $M(fx,gy,t)\ge\psi(\min\{M(x,y,t),M(x,fx,t),M(y,gy,t)\}),~\forall~x,y\in X,t>0.$ Then $f$ and $g$ have a unique common fixed point in $X.$
\end{cor}

We will later see that the above conclusion of the above corollary holds, in fact, in a complete GV fuzzy metric space.

\begin{cor}
\normalfont Let $(X,M,*)$ be a non-Archimedean weak $G$-complete fuzzy metric space and $f$ be a self-mapping on $X.$ Suppose for some $\psi\in\Psi,$ and a mapping $\beta:X^2\times(0,\infty)\to(0,\infty),$ $(f,f)$ is a pair of $\beta$-$\psi$-fuzzy contractive mappings such that

(i) $(f,f)$ is a symmetric pair of $\beta$-admissible mappings,

(ii) for some $x_0\in X,$ $\beta(x_0,fx_0,t)\le1,$ $\forall~t>0.$

(iii) for each sequence $(x_n)$ in $X$ with $\beta(x_n,x_{n+1},t)\le1,~\forall~n\in\mathbb N,t>0$ and a subsequence $(x_{r_n})$ of $(x_n)$ with $\lim\limits_{n\to\infty}x_{r_n}=x,$ we have $\beta(x_{r_n},x,t)\le1,~\forall~n\in\mathbb N,t>0.$

Then $f$ has a fixed point in $X.$ 

Moreover, if $\beta(x,y,t)\le1,~\forall~x,y\in X,t>0,$ then the fixed point is unique.
\end{cor}

In what follows, we show that the additional hypothesis: 

\textit{for each sequence $(x_n)$ in $X$ satisfying $\beta(x_n,x_{n+1},t)\le1,~\forall~n\in\mathbb N,~t>0,$ there exists $k_0\in\mathbb N$ such that $\beta(x_m,x_n,t)\le1,~\forall~m,n\in\mathbb N$ with $m\ge n>k_0$ and $t>0$}

\noindent enables the conclusion of Theorem \ref{th32} to work in a complete GV fuzzy metric space.

\begin{theo}\label{th33}
\normalfont Let $(X,M,*)$ be a non-Archimedean complete fuzzy metric space and $f,g$ be self-mappings on $X.$ Suppose for some $\psi\in\Psi,$ and a mapping $\beta:X^2\times(0,\infty)\to(0,\infty),$ $(f,g)$ is a pair of $\beta$-$\psi$-fuzzy contractive mappings such that

(i) $(f,g)$ is a symmetric pair of $\beta$-admissible mappings,

(ii) for some $x_0\in X,$ $\beta(x_0,fx_0,t)\le1,$ $\forall~t>0,$

(iii) for each sequence $(x_n)$ in $X$ with $\beta(x_n,x_{n+1},t)\le1,~\forall~n\in\mathbb N,t>0$ and $\lim\limits_{n\to\infty}x_n=x,$ we have $\beta(x_n,x,t)\le1,~\forall~n\in\mathbb N,t>0,$

(iv) for each sequence $(x_n)$ in $X$ satisfying $\beta(x_n,x_{n+1},t)\le1,~\forall~n\in\mathbb N,~t>0,$ there exists $k'\in\mathbb N$ such that $\beta(x_m,x_n,t)\le1,~\forall~m,n\in\mathbb N$ with $m\ge n>k'$ and $t>0.$

Then $f$ and $g$ have a common fixed point in $X.$ 

Moreover, if $\beta(x,y,t)\le1,~\forall~x,y\in X,t>0,$ then the fixed point is unique.
\end{theo}

\begin{proof}
\normalfont Proceeding as in Theorem \ref{res3}, we obtain the $G$-Cauchy sequence $(x_n)$ in $X$ given by $x_1=fx_0,x_2=gx_1,\cdots,x_{2n+1}=fx_{2n},x_{2n+2}=gx_{2n+1},\cdots$ such that $\beta(x_n,x_{n+1},t)\le1,~\forall~n=0,1,2,\cdots$ and $t>0.$ Thus by condition (iv), there exists $k_1\in\mathbb N$ such that $\beta(x_m,x_n,t)\le1,~\forall~m,n\in\mathbb N$ with $m\ge n>k_1$ and $t>0.$

We show that $(x_n)$ is a Cauchy sequence. If not, there exist $\epsilon\in(0,1),t>0$ and $k_0\in\mathbb N$ with $k_0>k_1$ such that for $k\in\mathbb N$ with $k\ge k_0,$ we have $m'(k),n(k)\in\mathbb N$ with $m'(k)>n(k)\ge k$ satisfying $M(x_{m'(k)},x_{n(k)},t)\le1-\epsilon.$

For each $k\ge k_0,$ we set $m(k)$ to be the smallest positive integer exceeding $n(k)$ such that $M(x_{m(k)},x_{n(k)},t)\le1-\epsilon\text{ and }M(x_{m(k)-1},x_{n(k)},t)>1-\epsilon.$ 

Therefore $\forall~k\ge k_0,$ 

$
1-\epsilon \ge M(x_{m(k)},x_{n(k)},t)\\
~~~~~~~~~\ge M(x_{m(k)-1},x_{n(k)},t)*M(x_{m(k)-1},x_{m(k)},t)\\
~~~~~~~~~\ge(1-\epsilon)*M(x_{m(k)-1},x_{m(k)},t).$

Taking limit as $k\to\infty,$ we obtain $\lim\limits_{k\to\infty} M(x_{m(k)},x_{n(k)},t)=1-\epsilon.$

Now for each $k\in\mathbb N$ with $k\ge k_0,$ we set

$u(k)=M(x_{m(k)},x_{m(k)+2},t)*M(x_{m(k)+2},x_{n(k)+2},t)*M(x_{n(k)},x_{n(k)+2},t),\\
~~~~~~~~~\text{if $m(k)=$ odd, $n(k)=$ even}\\
~~~~~~~~~=M(x_{m(k)},x_{m(k)+2},t)*M(x_{m(k)+2},x_{n(k)+1},t)*M(x_{n(k)},x_{n(k)+1},t),\\
~~~~~~~~~\text{if $m(k)=$ odd, $n(k)=$ odd}\\
~~~~~~~~~=M(x_{m(k)},x_{m(k)+1},t)*M(x_{m(k)+1},x_{n(k)+2},t)*M(x_{n(k)},x_{n(k)+2},t),\\
~~~~~~~~~\text{if $m(k)=$ even, $n(k)=$ even}\\
~~~~~~~~~=M(x_{m(k)},x_{m(k)+1},t)*M(x_{m(k)+1},x_{n(k)+1},t)*M(x_{n(k)},x_{n(k)+1},t),\\
~~~~~~~~~\text{if $m(k)=$ even, $n(k)=$ odd}$

\noindent and

$
v(k)=M(x_{m(k)},x_{m(k)+2},t)*\beta(x_{m(k)+1},x_{n(k)+1},t)M(x_{m(k)+2},x_{n(k)+2},t)*M(x_{n(k)},\\
~~~~~~~~~x_{n(k)+2},t),\text{ if $m(k)=$ odd, $n(k)=$ even}\\
~~~~~~~~~=M(x_{m(k)},x_{m(k)+2},t)*\beta(x_{m(k)+1},x_{n(k)},t)M(x_{m(k)+2},x_{n(k)+1},t)*M(x_{n(k)},\\ ~~~~~~~~~x_{n(k)+1},t),\text{ if $m(k)=$ odd, $n(k)=$ odd}\\
~~~~~~~~~=M(x_{m(k)},x_{m(k)+1},t)*\beta(x_{m(k)},x_{n(k)+1},t)M(x_{m(k)+1},x_{n(k)+2},t)*M(x_{n(k)},\\
~~~~~~~~~x_{n(k)+2},t),\text{ if $m(k)=$ even, $n(k)=$ even}\\
~~~~~~~~~=M(x_{m(k)},x_{m(k)+1},t)*\beta(x_{m(k)},x_{n(k)},t)M(x_{m(k)+1},x_{n(k)+1},t)*M(x_{n(k)},x_{n(k)+1},\\
~~~~~~~~~t),\text{ if $m(k)=$ even, $n(k)=$ odd}$

Since $\beta(x_m,x_n,t)\le1,~\forall~m,n\in\mathbb N$ with $m\ge n>k_0$, we have $\forall~k\ge k_0,$ $1-\epsilon\ge M(x_{m(k)},x_{n(k)},t)\ge u(k)\ge v(k)\cdots(*).$

Now for each $k\in\mathbb N$  with $k\ge k_0,$ we also set

$w(k)=\psi(\min\{M(x_{m(k)+1},x_{n(k)+1},t),M(x_{m(k)+1},x_{m(k)+2},t),M(x_{n(k)+1},x_{n(k)+2},t)\}),\\
~~~~~~~~~\text{if $m(k)=$ odd, $n(k)=$ even}\\
~~~~~~~~~=\psi(\min\{M(x_{m(k)+1},x_{n(k)},t),M(x_{m(k)+1},x_{m(k)+2},t),M(x_{n(k)},x_{n(k)+1},t)\}),\\
~~~~~~~~~\text{if $m(k)=$ odd, $n(k)=$ odd}\\
~~~~~~~~~=\psi(\min\{M(x_{m(k)},x_{n(k)+1},t),M(x_{m(k)},x_{m(k)+1},t),M(x_{n(k)+1},x_{n(k)+2},t)\}),\\
~~~~~~~~~\text{if $m(k)=$ even, $n(k)=$ even}\\
~~~~~~~~~=\psi(\min\{M(x_{m(k)},x_{n(k)},t),M(x_{m(k)},x_{m(k)+1},t),M(x_{n(k)},x_{n(k)+1},t)\}),\\
~~~~~~~~~\text{if $m(k)=$ even, $n(k)=$ odd}$

\noindent{and}

$w'(k)=\psi(\min\{M(x_{m(k)},x_{m(k)+1},t)*M(x_{m(k)},x_{n(k)},t)*M(x_{n(k)},x_{n(k)+1},t),\\
~~~~~~~~~M(x_{m(k)+1},x_{m(k)+2},t),M(x_{n(k)+1},x_{n(k)+2},t)\}),\text{if $m(k)=$ odd, $n(k)=$ even}
\\
=\psi(\min\{M(x_{m(k)},x_{m(k)+1},t)*M(x_{m(k)},x_{n(k)},t),M(x_{m(k)+1},x_{m(k)+2},t),\\
~~~~~~~~~M(x_{n(k)},x_{n(k)+1},t)\}),\text{if $m(k)=$ odd, $n(k)=$ odd}
\\
=\psi(\min\{M(x_{m(k)},x_{n(k)},t)*M(x_{n(k)},x_{n(k)+1},t),M(x_{m(k)},x_{m(k)+1},t),\\
~~~~~~~~~M(x_{n(k)+1},x_{n(k)+2},t)\}),\text{if $m(k)=$ even, $n(k)=$ even}
\\
=\psi(\min\{M(x_{m(k)},x_{n(k)},t),M(x_{m(k)},x_{m(k)+1},t),M(x_{n(k)},x_{n(k)+1},t)\}),\\~~~~~~~~~\text{if $m(k)=$ even, $n(k)=$ odd}.$

For each $k\in\mathbb N$ with $k\ge k_0,$ we further set

$r(k)=M(x_{m(k)},x_{m(k)+2},t)*w(k)*M(x_{n(k)},x_{n(k)+2},t), \text{if $m(k)=$ odd, $n(k)=$ even}\\
~~~~~~~~~=M(x_{m(k)},x_{m(k)+2},t)*w(k)*M(x_{n(k)},x_{n(k)+1},t), \text{if $m(k)=$ odd, $n(k)=$ odd}\\
~~~~~~~~~=M(x_{m(k)},x_{m(k)+1},t)*w(k)*M(x_{n(k)},x_{n(k)+2},t), \text{if $m(k)=$ even, $n(k)=$ even}\\
~~~~~~~~~=M(x_{m(k)},x_{m(k)+1},t)*w(k)*M(x_{n(k)},x_{n(k)+1},t), \text{if $m(k)=$ even, $n(k)=$ odd}$

\noindent\text{and}

$r'(k)=M(x_{m(k)},x_{m(k)+2},t)*w'(k)*M(x_{n(k)},x_{n(k)+2},t), \text{if $m(k)=$ odd, $n(k)=$ even}\\
~~~~~~~~~=M(x_{m(k)},x_{m(k)+2},t)*w'(k)*M(x_{n(k)},x_{n(k)+1},t), \text{if $m(k)=$ odd, $n(k)=$ odd}\\
~~~~~~~~~=M(x_{m(k)},x_{m(k)+1},t)*w'(k)*M(x_{n(k)},x_{n(k)+2},t), \text{if $m(k)=$ even, $n(k)=$ even}\\
~~~~~~~~~=M(x_{m(k)},x_{m(k)+1},t)*w'(k)*M(x_{n(k)},x_{n(k)+1},t), \text{if $m(k)=$ even, $n(k)=$ odd}.$

Since $\lim\limits_{k\to\infty}w'(k)=\psi(1-\epsilon),$ we have $\lim\limits_{k\to\infty}r'(k)=\psi(1-\epsilon).$

Using $(*),$ we have $1-\epsilon\ge M(x_{m(k)},x_{n(k)},t)\ge u(k)\ge v(k)\ge r(k)\ge r'(k),~\forall~k\ge k_0.$ Thus by taking limit as $k\to\infty,$ we have $1-\epsilon\ge\psi(1-\epsilon)>1-\epsilon,$ which leads to a contradiction.

Thus $(x_n)$ is a Cauchy sequence and hence converges to some $x\in X.$ 

We now show that $fx=gx=x.$

If $gx\ne x,$ then there exists $t_0>0$ such that $0<M(x,gx,t_0)<1.$ 

Since $M(x,gx,t_0)\ge M(x,x_{2n+1},t_0)*M(fx_{2n},gx,t_0)\\
~~~~~~~~~\ge M(x,x_{2n+1},t_0)*\beta(x_{2n},x,t_0)M(fx_{2n},gx,t_0)\\
~~~~~~~~~\ge M(x,x_{2n+1},t_0)*\psi(\min\{M(x,x_{2n},t_0),M(x,gx,t_0),M(x_{2n+1},x_{2n},t_0)\})$

so by taking limit as $n\to\infty$ we have $M(x,gx,t_0)\ge\psi(M(x,gx,t_0)>M(x,gx,t_0),$ which leads to a contradiction.

Thus $gx=x.$ Similarly, $fx=x.$ Hence $fx=gx=x.$

Uniqueness of the fixed point can be proved by proceeding as in Theorem \ref{res3}.
\end{proof}

\begin{cor}
\normalfont \normalfont Let $(X,M,*)$ be a non-Archimedean complete fuzzy metric space and $f$ be a self-mapping on $X.$ Suppose for some $\psi\in\Psi,$ and a mapping $\beta:X^2\times(0,\infty)\to(0,\infty),$ $(f,f)$ is a pair of $\beta$-$\psi$-fuzzy contractive mappings such that

(i) $(f,f)$ is a symmetric pair of $\beta$-admissible mappings,

(ii) for some $x_0\in X,$ $\beta(x_0,fx_0,t)\le1,$ $\forall~t>0,$

(iii) for each sequence $(x_n)$ in $X$ with $\beta(x_n,x_{n+1},t)\le1,~\forall~n\in\mathbb N,t>0$ and $\lim\limits_{n\to\infty}x_n=x,$ we have $\beta(x_n,x,t)\le1,~\forall~n\in\mathbb N,t>0,$

(iv) for each sequence $(x_n)$ in $X$ satisfying $\beta(x_n,x_{n+1},t)\le1,~\forall~n\in\mathbb N,~t>0,$ there exists $k_0\in\mathbb N$ such that $\beta(x_m,x_n,t)\le1,~\forall~m,n\in\mathbb N$ with $m\ge n>k_0$ and $t>0.$

Then $f$ has a fixed point in $X.$ 

Moreover, if $\beta(x,y,t)\le1,~\forall~x,y\in X,t>0,$ then the fixed point is unique.
\end{cor}

By putting $\beta\equiv1,$ in Theorem \ref{th33} we have the following:

\begin{cor}\label{cor35sa}
\normalfont Let $(X,M,*)$ be a non-Archimedean complete fuzzy metric space and $f,g$ be self-mappings on $X.$ Suppose for some $\psi\in\Psi,$ $$M(fx,gy,t)\ge\psi(\min\{M(x,y,t),M(x,fx,t),M(y,gy,t)\}),~\forall~x,y\in X,t>0.$$ Then $f$ and $g$ have a unique common fixed point in $X.$
\end{cor}

We encourage the readers to compare the above result with Corollary \ref{corts} and Theorem \ref{cort}.

\begin{theo}\label{res4}
\normalfont Let $(X,M,*)$ be a weak $G$-complete (KM) fuzzy metric space and $f,g$ be continuous self-mappings on $X.$ Suppose for some $\psi\in\Psi,$ and a mapping $\beta:X^2\times(0,\infty)\to(0,\infty),$ $M(x,y,t)>0\implies$ $$\beta(x,y,t)M(gfx,fgy,t)\ge\psi(\min\{M(x,y,t),M(x,fx,t),M(y,gy,t)\}),$$ $\forall~x,y\in X,t>0$ such that

(i) $(f,g)$ is a symmetric pair of $\beta$-admissible mappings,

(ii) for some $x_0\in X,$ $\beta(x_0,fx_0,t)\le1$ and $M(fx_0,gfx_0,t),M(gfx_0,fgfx_0,t)>0,$ $\forall~t>0.$

\noindent Then $f$ and $g$ have a common fixed point in $X.$ 

Moreover, if $\beta(x,y,t)\le1,~\forall~x,y\in X,t>0$ and for $x,y~(x\ne y)\in X,$ $M(x,y,t)>0,~\forall~t>0,$ then the fixed point is unique.
\end{theo}

\begin{proof}
\normalfont Define a sequence $(x_n)$ as follows: $x_1=fx_0,x_2=gx_1,\cdots,x_{2n+1}=fx_{2n},x_{2n+2}=gx_{2n+1},\cdots$.

Then proceeding as Theorem \ref{res3}, we obtain $$\beta(x_n,x_{n+1},t),\beta(x_{n+1},x_n,t)\le1,~\forall~n\in\mathbb N,t>0.$$

Let $n$ be a positive integer bigger than $1$. 

If $n\text{ is }\text{odd},$ then for chosen $t>0,$ $M(x_{n-1},x_n,t)>0\implies$ 

$M(x_{n+1},x_{n+2},t)=M(gfx_{n-1},fgx_{n},t)$

$\ge\beta(x_{n-1},x_{n},t)M(gfx_{n-1},fgx_{n},t)$

$\ge\psi(\min\{M(x_{n-1},x_{n},t),M(x_{n-1},fx_{n-1},t),M(x_{n},gx_{n},t)\})$

$=\psi(\min\{M(x_{n-1},x_{n},t),M(x_{n},x_{n+1},t)\})$

If $n$ is even, then for chosen $t>0,$ $M(x_{n-1},x_n,t)>0\implies$ 

$M(x_{n+1},x_{n+2},t)=M(gfx_{n},fgx_{n-1},t)$

$\ge\beta(x_{n},x_{n-1},t)M(gfx_{n},fgx_{n-1},t)$

$\ge\psi(\min\{M(x_{n},x_{n-1},t),M(x_{n},fx_{n},t),M(x_{n-1},gx_{n-1},t)\})$

$=\psi(\min\{M(x_{n-1},x_{n},t),M(x_{n},x_{n+1},t)\}).$

Consequently, $\forall~n\ge2,t>0,$ $M(x_{n-1},x_n,t)>0\implies$ $
\text{either } M(x_{n+1},x_{n+2},t)$ $\ge\psi
(M(x_{n-1},x_{n},t)
\text{ or }M(x_{n+1},x_{n+2},t)\ge\psi(M(x_{n},x_{n+1},t)).$

Since $M(x_1,x_2,t),M(x_2,x_3,t)>0,$ we conclude $M(x_{n-1},x_n,t)>0,~\forall~n\ge2,t>0.$ 

Let us set $S_t=\min\{M(x_1,x_2,t),M(x_2,x_3,t)\},~\forall~t>0.$ 

Then $M(x_{2n-1},x_{2n},t)\ge\psi^{n-1}(S_t)$ and $M(x_{2n},x_{2n+1},t)\ge\psi^{n-1}(S_t),~\forall~n\ge2,t>0.$

Now $\forall~t>0,S_t>0\implies\lim\limits_{n\to\infty}\psi^{n-1}(S_t)=1.$

Thus $\forall~t>0,$ $\lim\limits_{n\to\infty}M(x_{2n-1},x_{2n},t)=\lim\limits_{n\to\infty}M(x_{2n},x_{2n+1},t)=1,$ and consequently, $\lim\limits_{n\to\infty}M(x_{n},x_{n+1},t)=1.$

So $(x_n)$ is a $G$-Cauchy sequence in $X.$

Now proceeding as Theorem \ref{res3}, we conclude that $x$ is a common fixed point of $f$ and $g$ which is unique if $\beta(x,y,t)\le1,~\forall~x,y\in X,t>0$ and for $x,y~(x\ne y)\in X,$ $M(x,y,t)>0,~\forall~t>0.$

Hence the result follows.
\end{proof}

\begin{exm}
\normalfont We note that Example \ref{ex33} satisfies the hypothesis of Theorem \ref{res4}. 

Let us now replace the value of $\beta,$ in Example \ref{ex33}, with $1.$ It is clear that the modified example satisfies the hypothesis of Theorem \ref{res4} as well. 

We now show that the modified example does not satisfy the hypothesis of Theorem \ref{res3}. 

Suppose otherwise. Then by putting $x=\frac{1}{2},y=1,t=1$ in $$\beta(x,y,t)M(fx,gy,t)\ge\psi(\min\{M(x,y,t),M(x,fx,t),M(y,gy,t)\}),$$ we obtain $M\left(\frac{1}{2},1,1\right)\ge\psi\left(M\left(\frac{1}{2},1,1\right)\right),$ which is a contradiction since $\psi(x)>x,~\forall~x\in(0,1).$
\end{exm}

\begin{rem}
\normalfont We note that Example \ref{ex1} satisfies the hypothesis of first part of Theorem \ref{res4}. Since $f,g$ in the said Example have more than one common fixed points, the last condition is essential to ensure the uniqueness of the common fixed point in Theorem \ref{res4}. 
\end{rem}

\begin{cor}
\normalfont Let $(X,M,*)$ be a weak $G$-complete (KM) fuzzy metric space and $f$ be a continuous self-mapping on $X.$ Suppose for some $\psi\in\Psi,$ and a mapping $\beta:X^2\times(0,\infty)\to(0,\infty),$ $M(x,y,t)>0\implies$ $$\beta(x,y,t)M(f^2x,f^2y,t)\ge\psi(\min\{M(x,y,t),M(x,fx,t),M(y,fy,t)\}),$$ $\forall~x,y\in X,t>0$ such that

(i) $(f,f)$ is a symmetric pair of $\beta$-admissible mappings,

(ii) for some $x_0\in X,$ $\beta(x_0,fx_0,t)\le1$ and $M(fx_0,f^2x_0,t),M(f^2x_0,f^3x_0,t)>0,$ $\forall~t>0.$

\noindent Then $f$ has a fixed point in $X.$ 

Moreover, if $\beta(x,y,t)\le1,~\forall~x,y\in X,t>0$ and for $x,y~(x\ne y)\in X,$ $M(x,y,t)>0,~\forall~t>0,$ then the fixed point is unique.
\end{cor}

\begin{cor}
\normalfont Let $(X,M,*)$ be a weak $G$-complete (KM) fuzzy metric space and $f$ be a continuous self-mapping on $X.$ Suppose for some $\psi\in\Psi,$ and a mapping $\beta:X^2\times(0,\infty)\to(0,\infty),$ $M(x,y,t)>0\implies$ $$\beta(x,y,t)M(fx,fy,t)\ge\psi(\min\{M(x,y,t),M(x,fx,t)\},$$ $\forall~x,y\in X,t>0$ such that

(i) $(f,I)$ is a symmetric pair of $\beta$-admissible mappings ($I$ being the identity self-mapping on $X),$

(ii) for some $x_0\in X,$ $\beta(x_0,fx_0,t)\le1$ and $M(fx_0,f^2x_0,t)>0,$ $\forall~t>0.$

\noindent Then $f$ has a fixed point in $X.$ 

Moreover, if $\beta(x,y,t)\le1,~\forall~x,y\in X,t>0$ and for $x,y~(x\ne y)\in X,$ $M(x,y,t)>0,~\forall~t>0,$ then the fixed point is unique.
\end{cor}

\begin{proof}
\normalfont Immediate by setting $g=I$ in Theorem \ref{res4}.
\end{proof}

In the next theorem, we omit the continuity hypothesis of one of $f,g$ in Theorem \ref{res4}.

\begin{theo}\label{onlyf}
\normalfont Let $(X,M,*)$ be a weak $G$-complete (KM) fuzzy metric space and $f,g$ be self-mappings on $X$ such that one of them is continuous. Suppose for some $\psi\in\Psi,$ and a mapping $\beta:X^2\times(0,\infty)\to(0,\infty),$ $M(x,y,t)>0\implies$ $$\beta(x,y,t)M(gfx,fgy,t)\ge\psi(\min\{M(x,y,t),M(x,fx,t),M(y,gy,t)\}),$$ $\forall~x,y\in X,t>0$ such that

(i) $(f,g)$ is a symmetric pair of $\beta$-admissible mappings,

(ii) for some $x_0\in X,$ $\beta(x_0,fx_0,t)\le1$ and $M(fx_0,gfx_0,t),M(gfx_0,fgfx_0,t)>0,$ $\forall~t>0,$

(iii) for each sequence $(x_n)$ in $X$ with $\beta(x_n,x_{n+1},t)\le1,~\forall~n\in\mathbb N,t>0$ and a subsequence $(x_{r_n})$ of $(x_n)$ with $\lim\limits_{n\to\infty}x_{r_n}=x,$ we have $\beta(x_{r_n},x,t),\beta(x,x_{r_n},t)\le1,~\forall~n\in\mathbb N,t>0.$

Then $f$ and $g$ have a common fixed point in $X.$ 

Moreover, if $\beta(x,y,t)\le1,~\forall~x,y\in X,t>0$ and for $x,y~(x\ne y)\in X,$ $M(x,y,t)>0,~\forall~t>0,$ then the fixed point is unique.
\end{theo}

\begin{proof}
Proceeding as in Theorem \ref{res4}, we obtain the $G$-Cauchy sequence $(x_n)$ in $X$ given by $x_1=fx_0,x_2=gx_1,\cdots,x_{2n+1}=fx_{2n},x_{2n+2}=gx_{2n+1},\cdots$ such that $\beta(x_n,x_{n+1},t),\beta(x_{n+1},x_n,t)\le1,~\forall~n\in\mathbb N,t>0.$

Since $(X,M,*)$ is weak $G$-complete, there exists a subsequence $(x_{r_n})$ of $(x_n)$ such that $(x_{r_n})$ converges to some $x\in X$ where $r_n$'s are either all even or all odd.

Without loss of generality, suppose all the $r_n$'s are even.

\noindent\textit{Case I:} Let $f$ be continuous. Then $\forall~t>0,n\in\mathbb N,$ 

$M(x,fx,t)\ge M\left(x,x_{r_n+1},\frac{t}{2}\right)*M\left(x_{r_n+1},fx,\frac{t}{2}\right)\\=M\left(x,x_{r_n+1},\frac{t}{2}\right)*M\left(fx_{r_n},fx,\frac{t}{2}\right)\\=M\left(x,x_{r_n},\frac{t}{4}\right)*M\left(x_{r_n},x_{r_n+1},\frac{t}{4}\right)*M\left(fx_{r_n},fx,\frac{t}{2}\right)$

Taking limit as $n\to\infty,$ we see that the right hand side tends to $1.$ 

Thus we have $M(x,fx,t)=1,~\forall~t>0$ and consequently, $fx=x.$

We now show that $gx=x.$ Choose $t_1>0.$

Then for sufficiently large values of $n,$ $M(x,gx,t_1)\\\ge M(x,x_{r_n},\frac{t_1}{5})*M(x_{r_n},x_{r_n+1},\frac{t_1}{5})*M(x_{r_n+1},x_{r_n+2},\frac{t_1}{5})*M(x_{r_n+2},x_{r_n+3},\frac{t_1}{5})*M(x_{r_n+3},\\gx,\frac{t_1}{5})\\
\ge M(x,x_{r_n},\frac{t_1}{5})*M(x_{r_n},x_{r_n+1},\frac{t_1}{5})*M(x_{r_n+1},x_{r_n+2},\frac{t_1}{5})*M(x_{r_n+2},x_{r_n+3},\frac{t_1}{5})*M(fgx_{r_n+1},\\gfx,\frac{t_1}{5})\\
\ge M(x,x_{r_n},\frac{t_1}{5})*M(x_{r_n},x_{r_n+1},\frac{t_1}{5})*M(x_{r_n+1},x_{r_n+2},\frac{t_1}{5})*M(x_{r_n+2},x_{r_n+3},\frac{t_1}{5})*(\beta(x,x_{r_n+1},\\\frac{t_1}{5})M(gfx,fgx_{r_n+1},\frac{t_1}{5}))\text{ (since}\lim\limits_{n\to\infty}x_{r_n+1}=x,~\beta(x,x_{r_n+1},t)\le1,~\forall~n\in\mathbb N,t>0)\\
\ge M(x,x_{r_n},\frac{t_1}{5})*M(x_{r_n},x_{r_n+1},\frac{t_1}{5})*M(x_{r_n+1},x_{r_n+2},\frac{t_1}{5})*M(x_{r_n+2},x_{r_n+3},\frac{t_1}{5})*\\\psi(\min\{M(x_{r_n+1},x,\frac{t_1}{5}),M(x_{r_n+1},gx_{r_n+1},\frac{t_1}{5}),M(x,fx,\frac{t_1}{5})\})\\
\ge M(x,x_{r_n},\frac{t_1}{5})*M(x_{r_n},x_{r_n+1},\frac{t_1}{5})*M(x_{r_n+1},x_{r_n+2},\frac{t_1}{5})*M(x_{r_n+2},x_{r_n+3},\frac{t_1}{5})*\\\psi(\min\{M(x_{r_n+1},x,\frac{t_1}{5}),M(x_{r_n+1},x_{r_n+2},\frac{t_1}{5}),M(x,fx,\frac{t_1}{5})\}.$

Taking limit as $n\to\infty,$ we see that the right hand side tends to $1.$

Thus $M(x,gx,t_1)=1,~\forall~t_1>0\implies x=gx.$

\noindent\textit{Case II:} Let $g$ be continuous. Then we can similiarly show that $x$ is a common fixed point of $f,g.$

Again, if all the $r_n$'s are odd, it can be similarly shown that $fx=gx=x.$

Uniqueness of the fixed point can be proved by proceeding as in Theorem \ref{res3}.
\end{proof}

By putting $\beta\equiv1$ in Theorem \ref{onlyf} we have the following:

\begin{cor}\label{310}
\normalfont Let $(X,M,*)$ be a weak $G$-complete (KM) fuzzy metric space and $f,g$ be self-mappings on $X$ such that one of them is continuous. Suppose for some $\psi\in\Psi,$ $M(x,y,t)>0\implies$ $$M(gfx,fgy,t)\ge\psi(\min\{M(x,y,t),M(x,fx,t),M(y,gy,t)\}),$$ $\forall~x,y\in X,t>0$ such that  $M(fx_0,gfx_0,t),M(gfx_0,fgfx_0,t)>0,$ $\forall~t>0.$

Then $f$ and $g$ have a common fixed point in $X.$ 

Moreover, if for $x,y~(x\ne y)\in X,$ $M(x,y,t)>0,~\forall~t>0,$ then the fixed point is unique.
\end{cor}

We encourage the readers to compare the last result with Corollary \ref{corts} and Theorem \ref{cort}.

\begin{exm}
\normalfont Consider the fuzzy metric space $(X,M,\cdot)$ of Example \ref{ex33}. Let $f,g$ be self-mappings on $X$ such that $$fx=\begin{cases}1&\text{if }x\in\mathbb Q\\\noalign{\vskip9pt}\frac{1}{2}&\text{otherwise}\end{cases}.$$ 

and $gx=1,~\forall~x\in X.$

Clearly, $g$ is continuous on $X$ but $f$ is not. 

Since $gfx=fgy=1,~\forall~x,y\in X$ we have $$M(gfx,fgy,t)\ge\psi(\min\{M(x,y,t),M(x,fx,t),M(y,gy,t)\}),$$ $\forall~x,y\in X,t>0$ and $\psi\in\Psi.$

Thus by setting $\beta\equiv1$ we see that $f,g$ satisfy the hypothesis of Theorem \ref{onlyf}. 

Here $1$ is a common fixed point of $f$ and $g$ which is clearly unique.
\end{exm}

The following question is open:

\begin{ques}
\normalfont Can the continuity hypothesis of both $f$ and $g$ in Theorem \ref{onlyf} be omitted?
\end{ques}

In what follows, we show that the additional hypothesis: 

\textit{for each sequence $(x_n)$ in $X$ satisfying $\beta(x_n,x_{n+1},t)\le1,~\forall~n\in\mathbb N,~t>0,$ there exists $k_0\in\mathbb N$ such that $\beta(x_m,x_n,t)\le1,~\forall~m,n\in\mathbb N$ with $m\ge n>k_0$ and $t>0$}

\noindent enables the conclusion of Theorem \ref{onlyf} to work in a complete (KM) fuzzy metric space.

\begin{theo}
\normalfont Let $(X,M,*)$ be a non-Archimedean complete (KM) fuzzy metric space and $f,g$ be self-mappings on $X$ such that one of them is continuous. Suppose for some $\psi\in\Psi,$ and a mapping $\beta:X^2\times(0,\infty)\to(0,\infty),$ $M(x,y,t)>0\implies$ $$\beta(x,y,t)M(gfx,fgy,t)\ge\psi(\min\{M(x,y,t),M(x,fx,t),M(y,gy,t)\}),$$ $\forall~x,y\in X,t>0$ such that

(i) $(f,g)$ is a symmetric pair of $\beta$-admissible mappings,

(ii) for some $x_0\in X,$ $\beta(x_0,fx_0,t)\le1$ and $M(fx_0,gfx_0,t),M(gfx_0,fgfx_0,t)>0,$ $\forall~t>0,$

(iii) for each sequence $(x_n)$ in $X$ with $\beta(x_n,x_{n+1},t)\le1,~\forall~n\in\mathbb N,t>0$ and a subsequence $(x_{r_n})$ of $(x_n)$ with $\lim\limits_{n\to\infty}x_{r_n}=x,$ we have $\beta(x_{r_n},x,t),\beta(x,x_{r_n},t)\le1,~\forall~n\in\mathbb N,t>0,$

(iv) for each sequence $(x_n)$ in $X$ satisfying $\beta(x_n,x_{n+1},t)\le1,~\forall~n\in\mathbb N,~t>0,$ there exists $k_0\in\mathbb N$ such that $\beta(x_m,x_n,t)\le1,~\forall~m,n\in\mathbb N$ with $m\ge n>k_0$ and $t>0.$

Then $f$ and $g$ have a common fixed point in $X.$ 

Moreover, if $\beta(x,y,t)\le1,~\forall~x,y\in X,t>0$ and for $x,y~(x\ne y)\in X,$ $M(x,y,t)>0,~\forall~t>0,$ then the fixed point is unique.
\end{theo}

\begin{proof}
\normalfont Proceeding as in Theorem \ref{res4}, we obtain the $G$-Cauchy sequence $(x_n)$ in $X$ given by $x_1=fx_0,x_2=gx_1,\cdots,x_{2n+1}=fx_{2n},x_{2n+2}=gx_{2n+1},\cdots$ that will turn out to be Cauchy following an argument similar to Theorem \ref{th33}. Since $X$ is complete, there exists $x\in X$ such that $\lim\limits_{n\to\infty}x_n=x.$ Similar to Theorem \ref{onlyf}, we can show that $fx=gx=x.$ Uniqueness of the common fixed point can be proved by proceeding as in Theorem \ref{res3}.
\end{proof}

\begin{cor}
\normalfont Let $(X,M,*)$ be a non-Archimedean complete (KM) fuzzy metric space and $f$ be a continuous self-mapping on $X$. Suppose for some $\psi\in\Psi,$ and a mapping $\beta:X^2\times(0,\infty)\to(0,\infty),$ $M(x,y,t)>0\implies$ $$\beta(x,y,t)M(f^2x,f^2y,t)\ge\psi(\min\{M(x,y,t),M(x,fx,t),M(y,fy,t)\}),$$ $\forall~x,y\in X,t>0$ such that

(i) $(f,f)$ is a symmetric pair of $\beta$-admissible mappings,

(ii) for some $x_0\in X,$ $\beta(x_0,fx_0,t)\le1$ and $M(fx_0,f^2x_0,t),M(f^2x_0,f^3x_0,t)>0,$ $\forall~t>0,$

(iii) for each sequence $(x_n)$ in $X$ with $\beta(x_n,x_{n+1},t)\le1,~\forall~n\in\mathbb N,t>0$ and a subsequence $(x_{r_n})$ of $(x_n)$ with $\lim\limits_{n\to\infty}x_{r_n}=x,$ we have $\beta(x_{r_n},x,t),\beta(x,x_{r_n},t)\le1,~\forall~n\in\mathbb N,t>0$,

(iv) for each sequence $(x_n)$ in $X$ satisfying $\beta(x_n,x_{n+1},t)\le1,~\forall~n\in\mathbb N,~t>0,$ there exists $k_0\in\mathbb N$ such that $\beta(x_m,x_n,t)\le1,~\forall~m,n\in\mathbb N$ with $m\ge n>k_0$ and $t>0.$

Then $f$ has a fixed point in $X.$ 

Moreover, if $\beta(x,y,t)\le1,~\forall~x,y\in X,t>0$ and for $x,y~(x\ne y)\in X,$ $M(x,y,t)>0,~\forall~t>0,$ then the fixed point is unique.
\end{cor}


\begin{thebibliography}{99}

\bibitem{saadr1} Adhya, S., \& Deb Ray, A. (2022). On Weak $G$-Completeness for fuzzy metric spaces, Soft Computing, 22(5), 2099-2105.

\bibitem{saadr2} Adhya, S., \& Deb Ray, A. (2021). Some properties of Lebesgue fuzzy metric spaces, Sahand Communications in Mathematical Analysis, 18(1), 1-14.

\bibitem{ver94} George, A., \& Veeramani, P. (1994). On some results in fuzzy metric spaces. Fuzzy Sets and Systems, 64(3), 395-399.

\bibitem{gv} Gopal, D., \& Vetro, C. (2014). Some new fixed point theorems in fuzzy metric spaces. Iranian Journal of Fuzzy Systems, 11(3), 95-107.

\bibitem{g88} Grabiec, M. (1988). Fixed points in fuzzy metric spaces. Fuzzy sets and systems, 27(3), 385-389.

\bibitem{st} Gregori, V., Miñana, J. J., \& Morillas, S. (2012). Some questions in fuzzy metric spaces. Fuzzy Sets and Systems, 204, 71-85.

\bibitem{wg} Gregori, V., Miñana, J. J., \& Sapena, A. (2018). On Banach contraction principles in fuzzy metric spaces. Fixed Point Theory, 2018, vol. 19, num. 1, p. 235-248.

\bibitem{gs2} Gregori, V., \& Sapena, A. (2002). On fixed-point theorems in fuzzy metric spaces. Fuzzy sets and systems, 125(2), 245-252.

\bibitem{is} Istr\v{a}\c{t}escu, V., 1974. An introduction to theory of probabilistic metric spaces, with applications. Ed, Tehnica, Bucuresti.

\bibitem{km} Kramosil, I., \& Michálek, J. (1975). Fuzzy metrics and statistical metric spaces. Kybernetika, 11(5), 336-344.

\bibitem{men} Menger, K. (1942). Statistical metrics. Proceedings of the National Academy of Sciences of the United States of America, 28(12), 535.

\bibitem{mi} Mihet, D. (2008). Fuzzy $\psi$-contractive mappings in non-Archimedean fuzzy metric spaces. Fuzzy Sets and Systems, 159(6), 739-744.

\bibitem{s66} Sherwood, H. (1966). On the completion of probabilistic metric spaces. Zeitschrift für Wahrscheinlichkeitstheorie und Verwandte Gebiete, 6(1), 62-64.

\bibitem{ss} Schweizer, B., \& Sklar, A. (1960). Statistical metric spaces. Pacific J. Math, 10(1), 313-334.

\bibitem{ts} Turkoglu, D., \& Sangurlu, M. (2014). Fixed point theorems for fuzzy $\psi$–contractive mappings in fuzzy metric spaces. Journal of Intelligent \& Fuzzy Systems, 26(1), 137-142.

\bibitem{cv11} Vetro, C. (2011). Fixed points in weak non-Archimedean fuzzy metric spaces. Fuzzy Sets and Systems, 162(1), 84-90.

\end{thebibliography}
\end{document}